\documentclass[
amsmath,longbibliography,secnumarabic,floatfix,amssymb,nofootinbib,
nobibnotes,letterpaper,11pt,tightenlines,notitlepage,showkeys,
showlabels]{amsart}%{revtex4-1}

\usepackage[mode=image]{standalone} %
\usepackage[utf8]{inputenc} %
\usepackage[backend=bibtex]{biblatex} %

\usepackage{fancyhdr}
\usepackage[top=1in, left=1in, right=1in, bottom=1in]{geometry} %
\usepackage[svgnames]{xcolor} %
\usepackage{latexsym, amsmath, amscd, amsthm} %
\usepackage{mathrsfs} %
\usepackage{graphicx} %
\usepackage{mathtools} %
\usepackage{hyperref} %
\usepackage{algorithmicx} %
\usepackage{algpseudocode} %
\usepackage{enumerate} %
\usepackage{setspace} %
\usepackage{caption} %
\usepackage{subcaption} %

\newcommand{\KnotDia}{\mathcal{K}} %
\newcommand{\knotdia}{\kappa} %
\newcommand{\FlatKnotDia}{\mathscr{K}} %
\newcommand{\KnotShad}{\FlatKnotDia} %
\newcommand{\knotshad}{k} %
\newcommand{\knotgrowth}{\mu_K} %
\newcommand{\LinkDia}{\mathcal{L}} %
\newcommand{\linkdia}{\lambda} %
\newcommand{\LinkShad}{\mathscr{L}} %
\newcommand{\linkshad}{\ell} %
\newcommand{\MapClass}{\mathscr{M}} %

\newcommand{\PrimeLinkDia}{\mathcal{PL}} %
\newcommand{\primelinkdia}{p\lambda} %
\newcommand{\primelinkshad}{p\ell} %
\newcommand{\PrimeKnotDia}{\mathcal{PK}} %
\newcommand{\PrimeKnotShad}{\mathscr{PK}} %
\newcommand{\ArbSurf}{\Sigma} %
\newcommand{\Prb}{\mathbb{P}} %
\newcommand{\GraphOf}[1]{\Gamma(#1)} %

\DeclareMathOperator{\Aut}{aut} %

\newtheorem*{conjecture}{Conjecture} %
\newtheorem*{untheorem}{Theorem} %
\newtheorem{theorem}{Theorem} %
\newtheorem{proposition}[theorem]{Proposition} %
\newtheorem*{lemma*}{Lemma} %
\newtheorem{corollary}[theorem]{Corollary} %

\theoremstyle{definition} %
\newtheorem*{remark}{Remark} %
\newtheorem*{definition}{Definition} %

\begin{document}
%%% HEAD MATTER
\title[]{Asymptotic laws for random knot diagrams} \author{Harrison
  Chapman} \email{hchapman@math.uga.edu}
\address{Department of Mathematics\\
  University of Georgia, Athens GA}
% \noaffiliation{}
\date{\today}
%%%

\begin{abstract}
  We study random knotting by considering knot and link diagrams as decorated, (rooted) topological maps on spheres and pulling them uniformly from among sets of a given number of vertices \(n\), as first established in recent work with Cantarella and Mastin. The knot diagram model is an exciting new model which captures both the random geometry of space curve models of knotting as well as the ease of computing invariants from diagrams.

  We prove that unknot diagrams are asymptotically exponentially rare, an analogue of Sumners and Whittington's landmark result for self-avoiding walks. Our proof uses the same key idea: We first show that knot diagrams obey a pattern theorem, which describes their fractal structure. We examine how quickly this behavior occurs in practice. As a consequence, almost all diagrams are asymmetric, simplifying sampling from this model. We conclude with experimental data on knotting in this model. This model of random knotting is similar to those studied by Diao et al., and Dunfield et al.
\end{abstract}
\maketitle

%\doublespacing
%\onehalfspacing
\section{Introduction}
\label{sec:intro}

\thispagestyle{plain}

There are many models for sampling random knots: Self avoiding lattice
walks~\cite{Sumners_1988}, random space polygons
\cite{Cantarella2015b,Cantarella2013}, random Chebychev polynomials
\cite{Cohen2015}, random braid words~\cite{Nechaev1996}, \emph{Petaluma}
\cite{EvenZohar2014}, \textit{etc}. In this paper we will discuss the
\emph{random diagram model} introduced in~\cite{Cantarella2015} under which
\emph{knot diagrams} are drawn uniformly from the set of all diagrams with a
given number of crossings. The random diagram model exciting as it combines both
the geometry of space curve models of knotting with combinatorial ease of
computing knot invariants from diagrams. There has been some work on sampling
random diagrams~\cite{Diao2012,Dunfield2014talk}, but without regard to any
specific distribution. As well, alternating knot and link diagrams have been
studied~\cite{Jacobsen2002,Schaeffer2004} but alternating knot and link diagrams
are a significantly restricted subclass of diagrams in general.

Many models of random knots, specifically those which are used for the emulation
of polymer strands, obey the conjecture of Frisch, Wasserman, and Delbr\"uck
that almost all such objects are \emph{knotted} or \emph{non-trivial}. Indeed,
it is a landmark result of Sumners and Whittington~\cite{Sumners_1988} that this
is true for self-avoiding polygon models of knotting. Their proof was adapted to
numerous other models of random knotting, but it was as-of-yet unknown whether
this behavior is exhibited in the random diagram model.

In~\cite{Cantarella2015}, together with Cantarella and Mastin we provide a
tabulation of all knot diagrams with crossing number up to 10. In this paper we
begin by considering a slightly different object, \emph{rooted diagrams}, which
have no symmetries. We are then able to prove that in the large crossing number
limit, knot diagrams behave similarly to rooted diagrams, so that our results
carry over from the simpler rooted model to the original.

After introducing the key topological and combinatorial objects of
interest, we discuss the motivating parallels that this model of
knotting shares with those in the literature: \emph{Pattern theorems},
which assure that desired substructure appears often in fixed classes
of objects. We show that a pattern theorem result holds for many
different classes of diagrams, and provide constructions. We show then
how this result implies the key theorems of the paper---that knot
diagrams (rooted or unrooted) are both knotted and asymmetric.

We conclude with some experimental results. Experimentally, symmetries
disappear rapidly (\textit{i.e.}\ there are very few even in 10-crossing
diagrams) so the rooted and unrooted numerics are close in the
majority of our data set. We then examine the probabilities of
different knot types and compare the behavior to different models of
knotting. We finish with evidence that our numerics are somehow
\emph{different} than those of lattice polygon models of knotting in the sense
that we obtain different limiting ratios of knot type appearances
(\emph{c.f.}~\cite{Rensburg2011}).

\section{Definitions}
\label{sec:prelimdefs}

\subsection{Knots, links, and tangles}
\label{sec:knotlinktangledef}

A \emph{link} is an isotopy class of embeddings of one or more circles into
\(S^3\). A \emph{knot} is an isotopy class of embeddings of exactly one circle
into \(S^3\). Both of the prior are considered up to \emph{ambient isotopy} of
the embedded circles; roughly, this means that during the isotopy between two
knots, the knot ``cannot pass through itself.'' A \emph{knot diagram} (resp.
\emph{link diagram}) is a generic immersion (in that all intersection points are
double points) of a circle (resp.\ any number of circles) into the (in our cases
oriented) sphere \(S^2\) together with over-under (\textit{i.e.}\ sign)
information at each double point. The study of links and knots is well known to
be equivalent to the study of link diagrams and knot diagrams up to the
\emph{Reidemeister moves}, shown in Figure~\ref{fig:reidemeister}, by a theorem
of Reidemeister~\cite{Reidemeister1948}.
\begin{figure}[h!]
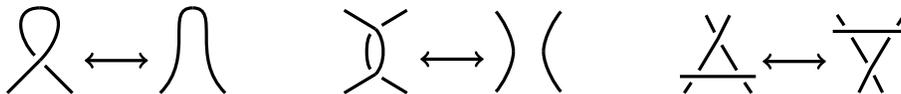

  \begin{tabular}{c@{\hspace{4em}}c@{\hspace{4em}}c}
    \centering
        \begingroup
        \input{#reidemeister_1}%
        \endgroup
    
    & %
        \begingroup
        \input{#reidemeister_2}%
        \endgroup
    
    & %
        \begingroup
        \input{#reidemeister_3}%
        \endgroup
    
  \end{tabular}
  \caption{The three Reidemeister moves.}
  \label{fig:reidemeister}
\end{figure}

The \emph{unknot} is the class of knots represented by the unit circle in
\(S^3\). A knot (or knot diagram) which represents the unknot is \emph{trivial}.
Knots and knot diagrams representing another class are \emph{knotted}. To
distinguish between space curves and their equivalence classes, we will use the
term \emph{knot type} to refer to a class of space curves, \emph{knots}.

A \emph{\(2k\)-tangle} is a generic immersion of \(k\) closed
intervals and any number of (possibly no) closed circles into \(B^3\)
so that precisely the \(2k\) interval ends all lie in the boundary. A
\emph{\(2k\)-tangle diagram} is a generic immersion of \(k\) closed
intervals and any number of closed circles into \(S^2\) together with
over-under information at each double point. In this paper, we will
only discuss tangle diagrams in which all \(2k\) ends of the intervals
lie in the same face of the sphere, so that the \(2k\)-tangle diagram
may be viewed as being an immersion into the disk \(D^2\) with exactly
the \(2k\) interval ends lying in the boundary circle. A \emph{strand}
of a tangle is any one circle or interval.
% \begin{figure}[h!]
%   \centering
%   \caption{A \(6\)-tangle diagram with 3 strands, and a \(6\)-tangle
%     diagram with 4.}
%   \label{fig:2ktangle}
% \end{figure}

\subsection{Topological maps}
\label{sec:topmapdefs}

Diagrams are considered up to ``embedded graph isomorphism.'' This
precisely means that the viewpoint we should have is that of
\emph{topological maps} on surfaces.

% If \(\ArbSurf\) is a surface, then its Euler characteristic
% \(\chi(\ArbSurf)\) is a topological invariant. The type \(g\) of a
% surface is defined by \(\chi(\ArbSurf) = 2 - 2g\) (for orientable
% surfaces, this definition agrees precisely with the genus \(g\)).

\begin{definition}
  A \emph{map with \(n\) vertices} \(M\) is a graph \(\GraphOf M\) embedded on a
  (usually oriented) surface \(\ArbSurf\) so that every connected component of
  \(\ArbSurf \setminus M\) is a topological disk. The connected components of
  \(\ArbSurf \setminus M\) are called the \emph{faces} of \(M\). If \(\ArbSurf\)
  is the oriented sphere, then the map \(M\) is \emph{planar}. As each face must
  be a disk, maps' underlying graphs are necessarily connected.

  A map \(M\) is \emph{4-regular}, \emph{4-valent}, or \emph{quartic} if every
  vertex in the underlying graph \(\GraphOf M\) has degree 4.
\end{definition}

Maps admit a decomposition \(M = (V, E, F)\) into vertices, edges, and faces,
which is equivalent to the data of a graph together with an embedding.

In this paper we will only consider planar maps, although by considering maps on
any oriented surface of arbitrary genus and applying the ideas of this work one
arrives at the study of \emph{virtual} diagrams.

Symmetry complicates the study of maps. A strategy to avoid this issue
is to \emph{root} the map by picking and directing a single edge:

\begin{definition}
  A \emph{rooted map} is a map together with a single edge marked with
  a direction, called a \emph{root edge}.

  An automorphism of a rooted map \(M\) would be required to fix the
  root edge, its direction, and the orientation of the surface near
  the root; hence \(\Aut(M)\) is the trivial group.
\end{definition}

\begin{figure}[h!]
  \centering %
        \begingroup
        \input{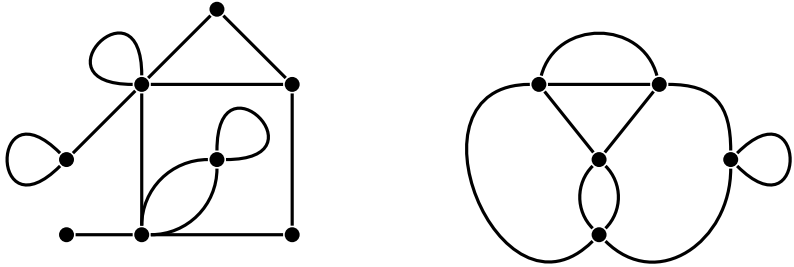}%
        \endgroup
    
  \caption{Two planar maps. The map on the right is in the class of
    knot (or link) shadows.}
  \label{fig:egplanarmaps}
\end{figure}

Maps have a well defined notion of \emph{dual map}; a map
\(M = (V, E, F)\) has dual \(M^* = (F, E^*, V)\), where there is an
edge \((f_1, f_2) \in E^*\) if \(f_1\) is adjacent to \(f_2\) in \(M\)
(faces are adjacent if they share an edge on their boundaries). The
dual graph of a 4-regular map is a \emph{quadrangulation}, \textit{i.e.}\ a map
for which every face has four bounding edges. A map is \emph{simple}
if it contains no parallel edges or self loops (its underlying graph
is simple). Given a rooted map \(M\), its dual is rooted as
follows. Let \(\rho\) be the root edge of \(M\) pointing from \(v_1\)
to the root vertex \(v_2\) be adjacent to the face \(f_1\) and the
root face \(f_2\). Then \((f_1, f_2)\) is the dual root edge and
directed from \(f_1\) to \(f_2\), and \(f_2\), \(v_2\) are the dual
root vertex and root face, respectively (and the dual of a dual rooted
map is the original rooted map).

Maps have a notion of substructure, which our pattern theorems will dictate:
\begin{definition}
  A map \(P\) with distinguished exterior face of \(k\)-edges is a
  \emph{submap} of a larger (possibly rooted) map \(M\) if there
  exists a cycle of \(k\) (possibly repeated) edges in \(M\) so that
  one of the two halves of \(M\) separated by the cycle is identical
  to \(P\).
\end{definition}

\subsection{Diagrams and shadows}
\label{sec:shadowdefs}

\begin{definition}
  A \emph{map decorated by a set \(S\)}, \((M, s)\) is a (possibly unrooted) map
  \(M = (V, E, F)\) together with a mapping \(s: V \to S\) which associates
  to each vertex of \(M\) an element of \(S\).
\end{definition}
We can now rephrase the definitions of diagrams using the vocabulary
of maps.
\begin{definition}
  A \emph{link shadow with \(n\) crossings} is a 4-regular planar map of \(n\)
  vertices. We will denote by \(\LinkShad_n\) the set of all \(n\)-crossing link
  shadows. Figure~\ref{fig:egplanarmaps} shows a general planar map, and a link
  shadow.

  A \emph{link diagram with \(n\) crossings} is a 4-regular
  planar map decorated with \(S = \{+,-\}\). This is equivalent
  to making a choice of over-under strand information at each
  vertex, as demonstrated in figure~\ref{fig:overunder}. We will denote the set of \(n\)-crossing link diagrams by
  \(\LinkDia_n\).
\end{definition}
Shadows and diagrams may be \emph{rooted} by taking additionally an edge
together with a choice of direction, as in Figure~\ref{fig:egrootedplanarmaps}.
From here on, maps, shadows, and diagrams will be assumed rooted unless
otherwise noted.

\begin{figure}[htbp]
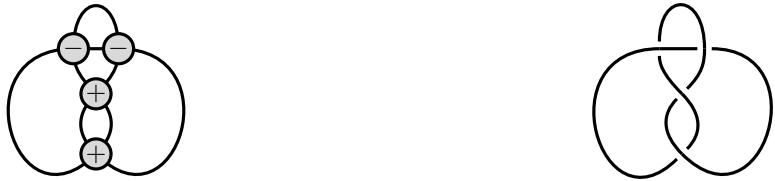

  \centering
  \begin{subfigure}[hbtp]{0.4\textwidth}
    \centering
        \begingroup
        \input{#figure-8-big-dshad}%
        \endgroup
    
    \caption{A figure-eight diagram viewed as a decorated shadow.}
  \end{subfigure}
  \hfil
  \begin{subfigure}[hbtp]{0.4\textwidth}
    \centering
        \begingroup
        \input{#figure-8-big}%
        \endgroup
    
    \caption{A figure-eight diagram viewed as a knot drawing.}
  \end{subfigure}
  \caption{After choosing once and for all a way to view signs as
    ``over-under'' information (\textit{i.e.}\ orientation around the
    knot), knot diagrams can be drawn as usual.}
  \label{fig:overunder}
\end{figure}

In order to be consistent with knot theory terminology, we will use
the word \emph{crossings} to refer to the vertices of shadows and
diagrams.

\begin{remark}
  This definition has then that unrooted link shadows and link
  diagrams are nearly the same objects as in the
  prequel~\cite{Cantarella2015}. There are two differences:
  \begin{enumerate}
  \item Our unrooted shadows (and diagrams) are on the \emph{oriented
      sphere}, as opposed to the unoriented sphere.
  \item Our unrooted shadows (and diagrams) do not come with any
    ``consistent'' choice of edge direction (\textit{i.e.}\ an orientation of
    link components).
  \end{enumerate}
  Notice that as there are at most 4 and at least 1 ``oriented shadows
  on the oriented sphere'' to each ``unoriented shadow on the
  unoriented sphere'', we get that asymptotically any results for our
  objects hold for those examined in the prequel.
\end{remark}

Indeed, \(\LinkShad\) is just another name for the class of 4-regular
planar maps counted by vertices; furthermore, the class of rooted
planar quadrangulations is dual to \(\LinkShad\).  Hence, the class
\(\LinkShad\) of link shadows has been counted
exactly~\cite{Brezin1978,Tutte1963}: If \(\linkshad_n = |\LinkShad_n|\),
then:
\[
  \linkshad_n = \frac{2(3^n)}{(n+2)(n+1)}\binom{2n}{n} \mathop{\sim}\limits_{n \to \infty} \frac{2}{\sqrt\pi}12^{n}n^{-5/2}.
\]
From this the exact counts of link diagrams can be determined as
well. If \(\linkdia_n = |\LinkDia_n|\), then
\[
  \linkdia_n = 2^n\linkshad_n =
  \frac{2^{n+1}(3^n)}{(n+2)(n+1)}\binom{2n}{n}
  \mathop{\sim}\limits_{n \to \infty} \frac{2}{\sqrt\pi}24^{n}n^{-5/2}.
\]

\begin{figure}[h!]
  \centering %
        \begingroup
        \input{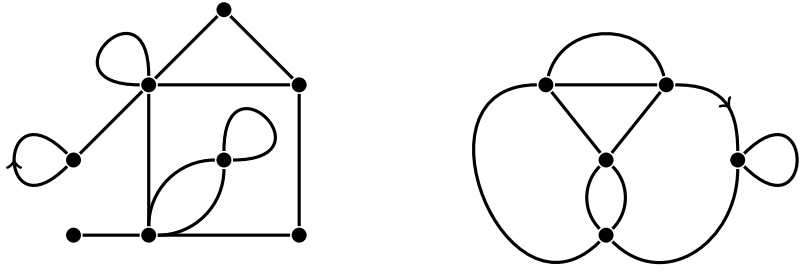}%
        \endgroup
    
  \caption{Two rooted planar maps. The map on the right is in the
    class of rooted knot shadows.}
  \label{fig:egrootedplanarmaps}
\end{figure}

Knot diagrams are the subset of link diagrams that have precisely one
``link component:''

\begin{definition}
  A \emph{link component} of a link shadow or
  diagram \(D\) is an equivalence class of edges modulo meeting across
  a vertex in \(D\).

  A \emph{knot shadow}\footnote{Kauffman calls these
  \emph{knot universes}~\cite{Kauffman87}.} is a link shadow which
  consists of precisely one link component. The class of knot shadows
  with \(n\) crossings is denoted by \(\KnotShad_n\).

  A \emph{knot diagram} is a link diagram which
  consists of precisely one link component. The class of knot diagrams
  with \(n\) crossings is denoted by \(\KnotDia_n\).
\end{definition}

Knot shadows \(\KnotShad_n\) represent a curious, small subclass of
\(\LinkShad_n\). Indeed, exact counts for
\(\knotshad_n = |\KnotShad_n|\) and \(\knotdia_n = |\KnotDia_n|\) are
not known except by experiments and
conjectures~\cite{Schaeffer2004}

\begin{conjecture}[Schaeffer-Zinn Justin 2004]
  There exist constants \(\knotgrowth\) and \(c\) such that
  \[
    \frac{\knotdia_n}{2^n} = \knotshad_n \mathop{\sim}\limits_{n \to \infty} c\knotgrowth^n \cdot n^{\gamma - 2},
  \]
  where
  \[
    \gamma = -\frac{1 + \sqrt{13}}{6},
  \]
  and \(\knotgrowth \approx 11.4\ldots\).
\end{conjecture}

This conjecture is of similar flavor to conjectures of the asymptotic counts of
self avoiding lattice walks, which are also of the
form~\cite{Madras2013,Rensburg2011}
\[
  a_n \mathop{\sim}\limits_{n\to\infty} c\lambda^n{}n^{\alpha}.
\]
Indeed, a large number of different combinatorial models exhibit this type of
growth; see Flajolet and Sedgewick~\cite{Flajolet2009} for examples.

Tangle diagrams may also be viewed in the language of maps:

\begin{definition}
  A \emph{\(2k\)-tangle shadow} is a map embedded on a sphere which is
  \(4\)-valent except for one distinguished ``external'' vertex of
  degree \(2k\). A \emph{\(2k\)-tangle diagram} is an unrooted
  \(2k\)-tangle shadow decorated (at non-exterior vertices) with signs
  \(\{+,-\}\).

  A tangle shadow (resp.\ diagram) \(T\) is \emph{contained} in a link
  shadow (diagram) \(D\) if the dual of \(T\), \(T^*\) (with exterior face the dual of
  the exterior vertex), is a submap of \(D^*\); for diagrams it is furthermore
  required that the signs of the crossings agree.
\end{definition}

We will view the exterior vertex of tangles as being ``at infinity''
so that tangle shadows and diagrams appear to be \(4\)-valent
decorated maps with loose exterior edges called \emph{arcs} or
\emph{legs}, as shown in Figure~\ref{fig:tangleex}.

\begin{figure}[hbtp]
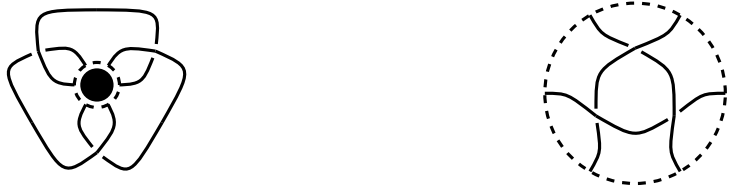

  \centering
  \begin{subfigure}[hbtp]{0.3\textwidth}
    \centering %
        \begingroup
        \input{#6-tangle-bdyvert}%
        \endgroup
    
  \end{subfigure}\hfil
  \begin{subfigure}[bhtp]{0.3\textwidth}
    \centering %
        \begingroup
        \input{#6-tangle}%
        \endgroup
    
  \end{subfigure}
  \caption{A 6-tangle diagram with boundary vertex (left) and boundary
    vertex viewed as disk boundary (right).}
  \label{fig:tangleex}
\end{figure}

Rooted diagrams (and shadows) may be viewed as directed \emph{long curves} or
\emph{two-leg diagrams} (resp. shadows) by cutting the root edge into two
directed half edges with one pointing towards its vertex (the \emph{hind leg})
and one away (the \emph{front leg}) as in Figure~\ref{fig:treflegs2}. Two-leg
diagrams are rooted 2-tangle diagrams.

% Rooted (knot or link) diagrams are equivalently viewed as
% \emph{two-leg diagrams} or \emph{2-tangle diagrams} as illustrated
% below.
\begin{figure}[h!]
  \centering
        \begingroup
        \input{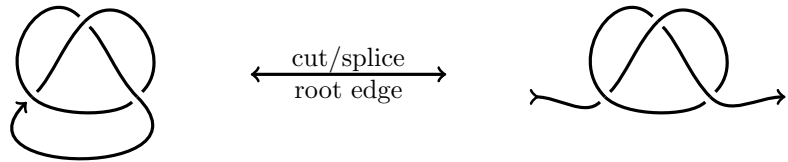}%
        \endgroup
    
  \caption{A rooted diagram of a trefoil, and its equivalent two-leg
  diagram}
  \label{fig:treflegs2}
\end{figure}

It will be handy to view diagrams in a manner similar to
pdcodes~\cite{Mastin2015} or combinatorial maps~\cite{Coquereaux16}. A diagram
with \(n\) vertices is, equivalently, a triple\footnote{This is a different
  decomposition of a map than the \((V, E, F)\) decomposition.} \((A, E, C)\) of arcs, edges, and
crossings where \(A\) consists of \(4n\) \emph{arcs} (sometimes called
\emph{half-edges} or \emph{flags}\footnote{The notation of flag indicates that
  arcs belong not only to a unique crossing and edge, but also to a unique face,
  although this observation is not necessary here.}),
\(2n\) unordered pairs of arcs (the edges), and \(n\) cyclic quadruples of arcs
together with a choice of sign (the crossings), up to renaming of arcs, so that
each arc appears in exactly \emph{one} edge and \emph{one} crossing. This
defines the embedding of a decorated graph into a surface; the diagram is planar
if this surface is the sphere. In this view, a root is a choice of arc. If \(a\)
is an arc in \(D\), define \(e(a)\) to be the unique edge containing \(a\) and
\(c(a)\) to be the unique crossing which contains \(a\).

We may view tangle diagrams in this manner; tangle diagrams are
triples \((A, E, C)\) where some ``exterior'' arcs \emph{do not belong
  to any edge} and \emph{live in the same face}.

% We can glue two $2k$-tangle diagrams to produce planar diagrams;
% \begin{definition}
%   Let $T$ and $U$ be $2k$-tangle diagrams for some \(k \ge 1\), and
%   say \(t_1\) and \(u_1\) are arcs in \(T\) and \(U\)
%   respectively. Let \((t_1,t_2,\hdots,t_{2k})\) be the exterior arcs
%   of \(T\) going counterclockwise and say \((u_1,u_2,\hdots,u_{2k})\)
%   are the exterior arcs of \(U\) going \emph{counterclockwise}. Then
%   \(T \cup_{t_1,u_1} U\) is the \emph{gluing of \(T\) to \(U\)}

%   % Then
%   % their dual maps $T^*$ and $U^*$ are quadrangulations with boundary
%   % $2k$-gons. Let \(\phi^*\) be an orientation-preserving
%   % identification of \(\partial T^*\) and \(\partial U^*\) taking
%   % vertices to vertices and edges to edges. Then
%   % \(\phi := {(\phi^*)}^*\) is a pairing of the loose arcs of \(T\) and
%   % \(U\) and the diagram
%   % \(T \cup_\phi U := {(T^* \cup_{\phi^*} U^*)}^*\) is the \emph{gluing
%   %   of \(T\) to \(U\)} along \(\phi\).
% \end{definition}

% Given a diagram $D$ with $2k$ loose arcs and a choice of two such arcs
% $a$ and $b$ which end in a common face, define the diagram produced by
% \emph{gluing} arcs $a$ and $b$ to be the diagram $D'$ with $2(k-1)$
% loose arcs which is identical to the diagram $D$ except with the arcs
% $a$ and $b$ joined into a single new edge. Given two diagram-arc pairs
% $(C,a)$ and \(D,b\), define the \emph{gluing} of \(C\) to \(D\) along
% $ab$ to be the diagram produced by joining the arcs $a$ and $b$ to
% produce a single edge.

A key strength of the framework proposed in this paper is the
simplicity with which the tools presented can be applied to other
interesting classes of diagram objects. To demonstrate this, we will
provide constructions for ``prime'' and ``reduced'' classes of
diagrams as well;
\begin{definition}
  A shadow or diagram \(D\) is \emph{prime} if it has more than \(1\)
  vertex and is not \(2\)-edge-connected, \textit{i.e.}\ there is no way to
  disconnect \(\GraphOf D\) by removing 2 edges. A shadow or diagram
  which is not prime is \emph{composite}. A rooted shadow or diagram
  is \emph{two-leg-prime} if it cannot be disconnected by removing two
  edges, \textit{one being the root edge}.

  A shadow or diagram \(D\) is \emph{reduced} if it has no
  disconnecting vertices (\textit{i.e.}, isthmi).
\end{definition}
It is important to note that prime knot diagrams can represent knot types which
are \emph{composite}. The condition of being a prime knot diagram is purely
graph-theoretic.

% We will denote by \(\PrimeLinkShad_n\) the set of prime link
% shadows, \(\PrimeKnotShad_n\) the set of prime knot shadows,
% \(\PrimeLinkDia_n\) the set of prime link diagrams,
% \(\PrimeKnotDia_n\) the set of prime knot diagrams, and
% \(\primelinkshad_n\), \(\primeknotshad_n\), \(\primelinkdia_n\), and
% \(\primeknotdia_n\) their respective cardinalities.

\begin{figure}[h!]
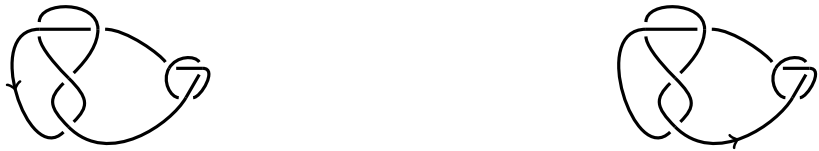
  \centering
  \begin{subfigure}[hbtp]{0.45\textwidth}
    \centering
        \begingroup
        \input{#41_31_cs_2lp}%
        \endgroup
    
    \caption{A rooting which is two-leg-prime.}
    \label{fig:2lprime}
  \end{subfigure}
  \hfil
  \begin{subfigure}[hbtp]{0.45\textwidth}
    \centering
        \begingroup
        \input{#41_31_cs_2lcomp}%
        \endgroup
    
    \caption{A rooting which is not two-leg-prime.}
    \label{fig:2lcomp}
  \end{subfigure}
  \caption{Depending on the rooting, a composite shadow may be
    two-leg-prime.}
  \label{fig:prime2legprime}
\end{figure}

Again, the counts of prime link shadows and prime link diagrams are
known precisely. Exact counts are known from their bijection with
simple quadrangulations~\cite{Albenque2014};
\[
  \frac{\primelinkdia_n}{2^n} = \primelinkshad_n = \frac{4(3n)!}{n!(2n + 2)!}.
\]
The counts for prime knot diagrams are again unknown.

\subsection{Composition of tangles}
\label{sec:comptangles}

Key arguments that we discuss in this paper involve the ability to
compose two diagrams in order to produce new diagrams with certain
properties. We will describe this by first discussing methods of
composition for tangles and then providing equivalences between
diagrams and tangles.

Given two tangle diagrams $T=(A_T,E_T,C_T), S=(A_S,E_S,C_S)$ with respective
exterior arcs $\{t_1,t_2,\hdots,t_{2k}\}$ and
$\{s_1,s_2,\hdots,s_{2\ell}\}$ and a collection of unordered pairings
\[\mu =
\{(t_{i_1},s_{j_1}),(t_{i_2},s_{j_2}),\hdots,(t_{i_r},s_{i_r})\}\]
where each \(t_i\) and \(s_i\) appears at most once, define the
\((2k+2\ell-2r)\)-tangle diagram \(T\#_\mu{}S\) to be the tangle diagram
\[ T\#_\mu{}S = (A_T\cup{}A_S,E_T\cup{}E_S\cup{}\mu,C_T\cup{}C_S). \]
Given an arbitrary choice of \(\mu\), it is possible that
\(T\#_\mu{}S\) is non-planar or that not all exterior arcs of
\(T\#_\mu{}S\) lie on the same face. We will define and use some
specific tangle compositions which avoid these pitfalls.

\subsubsection{Connect summation}
\label{sec:csdefn}

Analogous to the connect sum operation on knots in space, there is a
notion of \emph{diagram connect sum} on diagrams, described in
\cite{Cantarella2015}, definition 11. While the definition there is
technically for oriented diagrams and shadows (pd-codes are \textit{a priori}
oriented), connect sum for unoriented objects can be defined by first
picking an orientation for each summand. While there are four choices
in total while picking orientations, observe that there are only two
different ways to connect sum two diagrams.

\begin{definition}
  Given 2-tangle diagrams \(T, S\) with exterior arcs \(\{t_1,t_2\}\),
  \(\{s_1, s_2\}\) and a pair of arcs \((t_i,s_j)\) with
  \(i,j \in \{1,2\}\), there is a complementary pair of arcs
  \((t_{i'},s_{j'})\) with \(i \ne i'\) and \(j \ne j'\). So define
  the \emph{connect sum} of the head \(t_i\) of \(T\) to the tail
  \(s_j\) of \(S\);
  \[T\#_{(t_i,s_j)}S = T\#_{\{(t_i,s_j),(t_{i'},s_{j'})\}}S.\]

  Given a diagram \(D = (A, E, C)\) and an edge \(e = (ab)\) in \(D\), define
  the 2-tangle \(D \setminus e\) to be the 2-tangle diagram
  \((A, E\setminus \{e\}, C)\) with exterior arcs \(a\) and \(b\).

  A diagram and a choice of arc \((D,a)\) is equivalently a 2-tangle
  diagram \((D\setminus\{e(a)\}),a\). So given a
  2-tangle diagram and a choice of ``tail'' exterior arc \((T,b)\),
  define the \emph{connect sum} of \(T\) to \(D\) by \(a,b\) to be
  \[ D\#_{(a,b)}T = (D\setminus{}e(a))\#_{(a,b)}T. \]
\end{definition}
An example of this is given in Figure~\ref{fig:csexample}

\begin{figure}[hbtp]
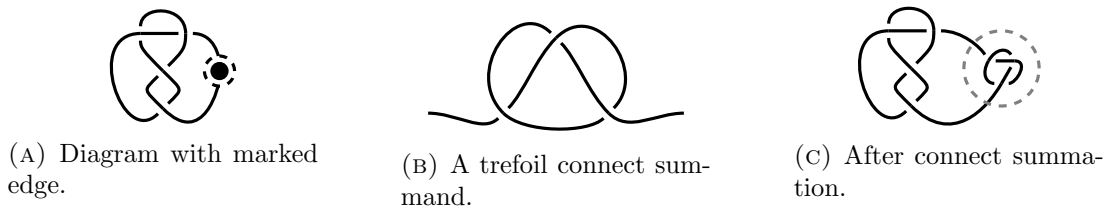

  \centering
  \begin{subfigure}[hbtp]{0.25\textwidth}
    \centering
        \begingroup
        \input{#figure-8-mkedge}%
        \endgroup
    
    \caption{Diagram with marked edge.}
  \end{subfigure}
  \hfil
  \begin{subfigure}[hbtp]{0.25\textwidth}
    \centering
        \begingroup
        \input{#trefoil_summand}%
        \endgroup
    
    \caption{A trefoil connect summand.}
    \label{fig:trefcs}
  \end{subfigure}
  \hfil
  \begin{subfigure}[hbtp]{0.25\textwidth}
    \centering
        \begingroup
        \input{#figure-8-cs-tref}%
        \endgroup
    
    \caption{After connect summation.}
  \end{subfigure}
  \caption{Connect sum of a trefoil into a figure-eight knot.}
  \label{fig:csexample}
\end{figure}

Link components behave predictably under connect summation: If
\(T, S\) have \(m,n\) closed circle components respectively, then
\(T\#S\) has exactly \(m+n-1\) closed circle components (independent of
choice of head and tail for the connect summation). It follows then
that if \(D\) is a knot diagram and \(T\) is a tangle diagram with no
closed circle components that \(D\#T\) is always itself a knot
diagram.

\subsubsection{Cyclic composition}
\label{sec:moretangcomp}

We can extend this definition to tangles with additional exterior
arcs, but we will then focus on tangles with precisely four.

\begin{definition}
  Given \(2k\)-tangle diagrams \(T, S\) with external arcs
  \[\{t_1,\hdots,t_{2k}\}, \{s_1,\hdots,s_{2k}\},\] choices
  \(t_{i_1}\) and \(s_{j_1}\) of arcs induce ordered tuples
  \(t_{i_1},\hdots,t_{i_{2k}}\) and \(s_{j_1},\hdots,s_{j_{2k}}\) by
  enumerating the external arcs of \(T\) counterclockwise and \(S\)
  clockwise, starting with the chosen arcs. Then with
  \(\mu = \{(t_{i_1},s_{j_1}),\hdots,(t_{i_{2k}},s_{j_{2k}})\}\),
  define the composition to be the link diagram
  \[ T\#_{(t_{i_1},s_{j_1})}S = T\#_\mu{}S. \]
  By this choice of \(\mu\), as \(T\) and \(S\) are both planar, so
  too is the composition \(T\#_\mu{}S\).
\end{definition}

The number of components in the resulting diagram will depend (in
addition to the count of closed loop components) on the ordering of
the external arcs of \(T, S\) as well as the precise matching. If
however we
are dealing with tangles of four exterior arcs, we can be
assured the existence of \(\mu\) which can guarantee some control over
the number of link components of the resultant diagram.

Consider the case where \(T, S\) are 4-tangle diagrams with no closed
loop components. For each tangle, color one strand red and one strand
blue. Suppose that \(T\) has its exterior arcs (ordered
counterclockwise) \(\{t_i\}_{i=1}^4\) colored
red-red-blue-blue. There exists some choice of composition of \(T\#S\)
so that \(T\#S\) consists of precisely one link component: Suppose
first then that \(S\) also has its arcs (ordered clockwise)
\(\{s_i\}_{i=1}^4\) colored red-red-blue-blue. Then the composition
\(T\#_{(t_1,s_1)}S\) consists of one link component. If instead \(S\)
had its arcs colored red-blue-red-blue, then any composition \(T\#S\)
can be seen to be precisely one link component.

Notice then that given a knot diagram \(D\), taking a crossing \(x\),
ignoring its sign, and designating it as the boundary vertex of a
tangle produces the 4-tangle denoted \(D \setminus x\) (there is some
abuse of notation here; we are actually removing both the crossing
\(x\), the four arcs in \(x\), and the edges in which those arcs
resided from \(D\)). Furthermore, as \(D\) is a knot diagram and
furthermore planar, \(D \setminus x\) must have its exterior arcs
colored red-red-blue-blue as in the case of \(T\) above. Hence given
any 4-tangle \(S\) there exists at least one way to compose
\((D \setminus x)\#S\) so that it is a planar knot diagram.

\section{Asymptotic structure theorems for diagrams}
\label{sec:structure}

\subsection{The Frisch-Wasserman-Delbr\"uck conjecture}
\label{sec:fwdconj}

The study of \emph{random} knotting arises in numerous areas, principal among
which is polymer physics: Polymers (such as DNA or proteins) are
considered to be strings in space and in many cases their function (or
lack thereof) depends on any ``knots'' that appear
within~\cite{Rawdon2012,Rawdon2015}. The \emph{random diagram model}
of random knotting is then: Given a number \(n\) of crossings, sample
uniformly an unlabeled knot diagram with \(n\) crossings and return
its knot type. It is similar to models of~\cite{Diao2012}
and~\cite{Dunfield2014talk}, but these models do not sample from any
well-understood measure on spaces of knot diagrams.

In the context of DNA topology, Frisch and Wasserman~\cite{Frisch1961} and
Delbr\"uck~\cite{Delbruck1962} independently conjectured;
\begin{conjecture}[Frisch-Wasserman 1962, Delbr\"uck 1961]
  As the size \(n\) of a randomly sampled knot grows large, the
  probability that it is knotted tends to 1.
\end{conjecture}
The first proof of the conjecture was for \(n\)-step self-avoiding lattice
polygons, a landmark result by Sumners and Whittington~\cite{Sumners_1988}:
\begin{untheorem}[Sumners-Whittington 1988]
  As the number of steps \(n\) of a self-avoiding lattice polygon
  grows large, the probability that the polygon is knotted tends to 1
  exponentially quickly.
\end{untheorem}
Shortly thereafter the conjecture was proved in view of other models
of space curves; self-avoiding Gaussian polygons
(\cite{Douglas1994}), self-avoiding equilateral polygons
(\cite{Diao1995}), \textit{etc.}

As mentioned, the primary purpose of this work is to ascertain that the
Frisch-Wasserman-Delbr\"uck (FWD) conjecture holds in our model;
\begin{theorem}
  \label{thm:knotted}
  As the number of crossings \(n\) of a randomly sampled knot diagram
  grows large, the probability that the diagram is knotted tends to 1
  exponentially quickly.
\end{theorem}
This result will follow from the fractal structure of knot diagrams. We will
prove results for rooted knot diagrams which extend (asymptotically) to the
unrooted case as there are always at least 1 and at most \(4n\) rooted diagrams
to every unrooted diagram. Asymptotically, any \emph{numerical} results on
rooted diagrams apply up to a factor of \(4n\) for the unrooted diagrams as;

\begin{theorem}
  \label{thm:asymm}
  As the number of crossings \(n\) of an (unrooted) knot diagram grows
  large, the probability that the diagram has a nontrivial
  automorphism group tends to 0 exponentially quickly.
\end{theorem}
These two results answer two experimentally motivated questions posed
in~\cite{Cantarella2015} in the affirmative. Indeed,
Theorem~\ref{thm:asymm} suggests that, for large \(n\), experiments
(\textit{c.f.}\ Section~\ref{sec:randres}) for unrooted knot diagrams
can be run instead on rooted knot diagrams and results will differ
only in that the rooted diagrams \(4n\)-to-1 cover unrooted
diagrams. While sampling rooted knot diagrams uniformly is still
nontrivial, it is reasonably quick to generate rooted knot diagrams of
70 crossings (but it as of yet computationally infeasible to tabulate
even all 12-crossing unrooted diagrams).

\subsection{The Pattern Theorem}
\label{sec:patternthm}

Sumners and Whittington's proof of the FWD
conjecture for self-avoiding lattice polygons makes use of Kesten's
pattern theorem~\cite{Kesten1963,Kesten1964,Madras2013} which states that
patterns---short walk configurations---appear linearly often in long
self-avoiding walks.

\subsubsection{Attachment}
\label{sec:attachments}

We make use of a similar strategy: Theorem 2 of Bender \textit{et
al}.~\cite{Bender1992} provides a pattern theorem for maps, provided
a strategy of attaching a desired pattern. However, care is required
in the case of knot or link \emph{diagrams}, our \emph{decorated} maps
with extra structure. We thus retrace the proof while keeping this in
mind. Say that if \(D\) is a link diagram, then \(e(D)\) is the number
of edges in the diagram \(D\). The proof depends on the existence of
an ``attachment'' scheme:
\begin{definition}
  Let \(Q\) be a tangle diagram, \(\MapClass\) some class of diagrams,
  and \(\mathscr{H}\) a subclass of \(\MapClass\). A \emph{viable
    attachment} of the tangle \(Q\) into \(\mathscr{H}\) is a method
  of taking a diagram \(D \in \mathscr{H}\) and producing new
  diagrams \(D'\) that contain \(Q\) as a subtangle satisfying:
  \begin{enumerate}
  \item For some fixed positive integer \(k\), at least
    \(\lfloor e(D)/k \rfloor\) possible non-conflicting places of
    attachment exist. This means that at least
    \(\lfloor e(D)/k \rfloor\) attachment operations (of \(Q\)) to a
    single diagram can be ``parallelized'' and all performed at once
    to a diagram to produce a new diagram with at least
    \(\lfloor e(D)/k \rfloor\) copies of \(Q\) contained as subtangles.
  \item Only diagrams in \(\MapClass\) are produced (\(\MapClass\) is
    closed under attachments of \(Q\)).
  \item For any diagram produced as such we can identify the copies of
    \(Q\) that have been added and they are all pairwise vertex
    disjoint. Identifying tangles in diagrams is trivial since a
    diagram contains a tangle if there is some dual cycle whose
    interior is the tangle itself. We will consider only \(Q\) which
    are \emph{always} pairwise disjoint, and hence any attachment
    would satisfy the latter half of this condition.
  \item Given the copies of \(Q\) that have been added, the original
    diagram and associated places of attachment are uniquely
    determined. For our attachments, we will provide suitable
    ``inverse'' operations which are themselves parallelizable (since
    instances of \(Q\) are disjoint). This condition then follows.
  \end{enumerate}
\end{definition}

Our tangle compositions yield then the two attachment schemes which
are relevant to the examples that follow.

\begin{definition}
  Consider the 2-tangle of one strand \(Q\) which is two-leg-prime and
  asymmetric with tail arc \(b\). Let \(\MapClass\) be one of:
  \begin{enumerate}
  \item all (rooted) link diagrams,
  \item all (rooted) knot diagrams, or
  \item provided \(Q\) is reduced, all (rooted) reduced knot or link
    diagrams.
  \end{enumerate}
  Then define the attachment \emph{edge replacement} by, given a
  diagram \(D\) and any arc \(a\) in \(D\), defining \(D'\) to be the
  new diagram \(D' = D \#_{(b,a)} Q\).
\end{definition}

\begin{proposition}
  Edge replacement is a viable attachment for choices of \(\MapClass\)
  enumerated above.
  \label{thm:edgerep_viable}
\end{proposition}
\begin{proof}
  We consider each required property.
  \begin{enumerate}
  \item It has at least \(n-1 \ge \lfloor n/2 \rfloor\) places of
    attachment to a rooted diagram \(D\) with \(n\) edges;
    non-conflicting ways to connect sum \(Q\) into \(D\) are precisely
    the number of non-root edges of \(D\). Given a collection
    \(\{a_i\}_{i=1}^r\) of different arcs, none of whom share an edge,
    in \(D\), edge replacements performed in any order will produce
    the same diagram (since edge replacement does not alter any arcs
    or edges in \(D\) besides the arc chosen and its edge).

  \item Depending on the choice of tangle \(Q\) and class of diagrams
    \(\MapClass\), diagram connect sum can be shown to keep diagrams
    in \(\MapClass\). Namely, provided a tangle \(Q\) of one strand,
    this attachment does not change the number of link components of
    diagrams. Furthermore, If \(Q\) is reduced, then the introduction
    of \(Q\) into a diagram by connect summation introduces no new
    disconnecting vertices.

  \item We can identify the copies of \(Q\) inside the resultant
    diagram by our definition of what it means for a tangle to be
    included in a diagram. Provided \(Q\) is two-leg-prime,
    occurrences of \(Q\) must be pairwise disjoint:

    Suppose to the contrary that \(Q_1\) and \(Q_1\) are instances of
    \(Q\) contained subtangles in a diagram \(D\) which are \emph{not}
    disjoint. The definition of being a subtangle means there exists
    some dual cycle of edges \((e_1, e_2)\) which isolates the tangle
    \(Q_1\) from the rest of \(D\); the cycle is necessarily of length
    2 as \(Q\) is a 2-tangle.

    Now, the condition that \(Q_1\) and \(Q_2\) are not disjoint means
    that, without loss of generality, the edge \(e_2\) appears in
    the interior of \(Q_2\). However, \(Q_2\setminus e_2\) is necessarily
    disconnected, which contradicts that \(Q\) was two-leg prime. So
    instances of \(Q\) must be pairwise disjoint.

  \item Edge replacement is an invertible operation. Given an instance
    of \(Q\) in \(D\) which is identified by the dual edge cycle
    \((e_1, e_2)\), we can recover the diagram \(D'\): Without loss of
    generality, as \(Q\) is asymmetric, we can assume that
    \(e_1 = (ab)\) is the edge which connects the tail arc \(b\) of
    \(Q\) to the arc \(a\) of \(D\). Removing \(e_1\) and \(e_2\) from
    \(D\) yields the pair of 2-tangles \(D^o\) and \(Q\). Denote by
    \(c\) the arc of \(D^o\) which is not \(a\). Then \(D'\) is the
    tangle \(D^o\) together with additional closing edge
    \(e_3 = (ac)\) and we have that \(D = D' \#_{(a,b)} Q\). We made
    no choices in this procedure beyond choice of instance of \(Q\)
    inside \(D\) and each step was entirely local.
  \end{enumerate}
  Hence, edge replacement is a viable attachment.
\end{proof}
If instead we have a \(2k\)-tangle diagram \(Q\) (for \(k>1\)) we can
still define an attachment into diagrams by first joining together
pairs of all but two exterior edges to produce a \(2\)-tangle
\(Q\). This 2-tangle \(Q\) can then be attached by edge replacement.

For classes of prime diagrams which are not closed under connect
summation, we have an alternate attachment;

\begin{definition}
  Consider the nontrivial 4-tangle \(Q\) that is asymmetric and at
  least 4-edge-connected with designated arc \(b\). Let \(\MapClass\)
  be either the class of prime link diagrams or (if \(Q\) consists of
  precisely two strands) the class of prime knot diagrams. Then define
  the \emph{crossing replacement} of a given diagram \(D\) and an arc
  \(a\) in \(D\) to be the diagram \(D' = D\#_{(a,b')} P_\sigma\),
  where \(\sigma\) is the sign of the crossing \(c(a)\) in \(D\) and
  \(P_\sigma\) is defined in figure~\ref{fig:psigmadef}.

  \begin{figure}[hbtp!]
    \centering
        \begingroup
        \input{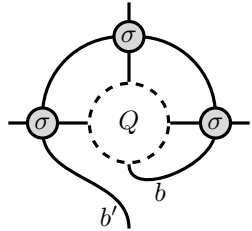}%
        \endgroup
    \caption{Definition of the ``capsid'' tangle \(P_\sigma\).}
    \label{fig:psigmadef}
  \end{figure}
\end{definition}

\begin{proposition}
  Crossing replacement is a viable attachment for choices of
  \(\MapClass\) enumerated above.
  \label{thm:xrep_viable}
\end{proposition}
\begin{proof}
  We consider each required property.
  \begin{enumerate}
  \item For the class of all prime link diagrams, any choice of arc in
    a diagram provides a crossing attachment of \(Q\) which is a link
    diagram. For the class of all prime knot diagrams on the other
    hand, each crossing has at least one choice of arc for which the
    insertion of \(P_\sigma\) produces a knot diagram. So for any
    diagram and for each crossing inside the diagram, there exists at
    least one viable location for crossing replacement, and provided
    crossing replacements happen at different crossings, they are
    independent. So there exists at least
    \(n/2 \ge \lfloor n/2 \rfloor \) non-conflicting places of
    attachment of \(Q\) into a diagram with \(n\) edges.

  \item Provided appropriate \(Q\), the attachment is closed in
    \(\MapClass\) as we only choose appropriate arcs.

  \item We can identify the copies of \(P_\sigma\) inside the resultant
    diagram by our definition of what it means for a tangle to be
    included in a diagram. Suppose that \(P_1\) and \(P_2\) are
    instances of \(P_\sigma\) which are not disjoint in \(D\)
    (\(\sigma\) does not matter and can be different for each
    \(P_i\)).

    That \(P_1\) and \(P_2\) are not disjoint means that the boundary
    cycle \((e_1, e_2, e_3, e_4)\) of \(P_1\) must in some part enter
    the interior of the instance \(P_2\). Notice that as \(Q\) is
    nontrivial, there is no dual path of length 1 through \(P_2\);
    this implies that the subpath which is inside \(P_2\) cannot be
    length 1. Conversely, it cannot be length 3 (switch the roles of
    \(P_1\) and \(P_2\)). So the subpath would have to be of length 2;
    as the dual map is a quadrangulation, a dual path of length two
    would have to be trivial (which cannot happen by structure of
    \(P_i\)) to begin at one face and end at the face
    opposite. However, by structure of \(P_i\) and that \(Q\) is
    nontrivial this is impossible.

  \item Crossing replacement is an invertible operation. Given an
    instance of \(P_\sigma\) in \(D\) which is identified by the dual
    edge cycle \((e_1, e_2, e_3, e_4)\), we can recover the diagram
    \(D'\): Without loss of generality, as \(P_\sigma\) is asymmetric,
    we can assume that \(e_1 = (ab')\) is the edge which connects the
    tail arc \(b'\) of \(P_\sigma\) to the arc \(a\) of
    \(D\). Removing \(e_1, e_2, e_3\), and \(e_4\) from \(D\) yields
    the pair of 4-tangles \(D^o\) and \(P_\sigma\). Denote by
    \(x,y,z\) the arcs counterclockwise around \(D^o\) after
    \(a\). Then \(D'\) is the tangle \(D^o\) together with the four
    arcs brought together into a crossing \(x\) with sign \(\sigma\)
    identified from \(P_\sigma\). We have that
    \(D = D' \#_{(a,b')} P_\sigma\). We made no choices in this
    procedure beyond choice of instance of \(P_\sigma\) inside \(D\)
    and each step was entirely local.
  \end{enumerate}
  Hence, crossing replacement is a viable attachment.
\end{proof}

\subsubsection{Pattern theorems for diagrams}

Given a generating function \(F\), let \(r(F)\) be its radius of
convergence. Being able to attach desired subtangles freely
yields;
\begin{theorem}Let \(\MapClass\) be some class of link diagrams on any
  surface and let \(Q\) be a tangle diagram that can be contained in
  diagrams in \(\MapClass\). Let \(M(x)\) be the generating function
  by number of edges for \(\MapClass\). Let \(H(x)\) be the generating
  function by number of edges for those link diagrams \(M\) in
  \(\MapClass\) that contain less than \(ce(M)\) pairwise disjoint
  copies of \(Q\); call this class \(\mathscr{H}\). Suppose there is a
  viable attachment for \(Q\) into \(\mathscr{H}\).  If \(0 < c < 1\)
  is sufficiently small, then \(r(M) < r(H)\). The diagrams may be
  rooted or not.
  \label{thr:weakpattern}
\end{theorem}

The proof of the theorem makes use of a lemma of Bender \textit{et al.}:

\begin{lemma*}[\cite{Bender1992}, Lemma 3]
  If
  \begin{enumerate}
  \item \(F(z) \ne 0\) is a polynomial with non-negative coefficients
    and \(F(0) = 0\),
  \item \(H(w)\) has a power series expansion with non-negative
    coefficients and \(0 < r(H) < \infty\),
  \item for some positive integer \(k\) the linear operator
    \(\mathscr{L}\) is given by \(\mathscr{L}(w^n) =
    z^n{(F(z)/z)}^{\lfloor n/k \rfloor}\), and
  \item \(G(z) = \mathscr{L}(H(w))\),
  \end{enumerate}
  then \({r(H)}^k = {r(G)}^{k-1}F(r(G))\).
\end{lemma*}

The proof of the theorem then remains almost unchanged from the
original theorem. Depending on the class \(\MapClass\) of diagrams and
tangles \(P\) considered, the construction of the attachment will
change.

\begin{proof}[Proof of Theorem~\ref{thr:weakpattern}.]
  Let \(G(z)\) be the generating function which counts the ways by edges of
  attaching some number between \(0\) and \(\lfloor n/k \rfloor\)
  copies of \(Q\) to diagrams in \(\MapClass\) counted by
  \(H(x)\). The method of attachment leads to the relation
  \(G(z) = \mathscr{L}(H(w))\), with \(\mathscr{L}\) as defined in the
  lemma, where \(F(z) = z + z^q\) and \(q\), is the number of edges
  added when a copy of \(Q\) is attached, as
  \[
    G(z) = \sum_{Y \in \mathscr{H}}{z^{e(Y)}{ \left( 1 + z^{q-1} \right)}^{\lfloor e(Y)/k \rfloor}} = \mathscr{L}(H(w)).
  \]
  Let \(g_n\) be the coefficients of \(G(z)\).

  Suppose \(M \in \MapClass\) contains \(m\) copies of \(Q\). By
  property (3) of our attachment, \(m \le n\). If \(M\) had been
  produced from some diagram \(K\) in \(\mathscr{H}\) by our
  attachment process, we can find all possible \(K\) by removing at
  least \(m - cn\) copies of \(Q\) from \(M\). It is possible to bound
  from above the number of ways to do this by

  \begin{align*}
    \sum_{j \ge m-cn}{\binom m j} &= \sum_{k < cn}{\binom m k} < \sum_{k < cn}{\binom n k} \le {n\binom n{cn}} \\
    &\le \frac{n{(ne)}^{cn}}{{cn}^{cn}} = n{\left(\frac ec\right)}^{cn} =: t_n.
  \end{align*}

  If \(M(x) = \sum{m_n{x^n}}\), then \(m_n \ge g_n\) and \(t_n > 1\)
  for sufficiently large \(n\), so \(m_n \ge g_n/t_n\). Hence,
  \[
    1/r(M) \ge \limsup_{n \to \infty}{{(g_n/t_n)}^{1/n}} = \lim_{n \to
      \infty}{{(t_n)}^{-1/n}}\limsup_{n \to \infty}{{(g_n)}^{1/n}} \ge
    {(c/e)}^c/r(G).\footnote{The power of \(1/n\) on \(\limsup_{n \to
        \infty}{{(g_n)}}\) is missing in the original proof of Bender
          \textit{et al.}}
  \]
  By the prior lemma, \({r(H)}^k = {r(G)}^k(1 + {r(G)}^{q-1})\) so
  that
  \[ r(H)/r(M)\ \ge {(1 + {r(G)}^{q-1})}^{1/k} {(c/e)}^c. \]
  As there are fewer than \(12^n\) planar maps with \(n\)
  edges~\cite{Bender1986}, a trivial bound on the coefficients
  \(h_n\) of \(H(z)\) is \(h_n \le (12^n)(\kappa^{n/2})\) where
  \(\kappa \ge 1\) is the size of the set of decorations on the
  vertices of the diagram (in the scope of this paper,
  \(\kappa = 2\) or 1).  This provides the bound
  \({r(H)}^k \ge 1/{(12\sqrt \kappa)}^k\).  Then, as
  \(\lim_{c \to 0^+}{{(c/e)}^c} = 1\) and
  \({r(G)}^k(1 + {r(G)}^{q-1}) = {r(H)}^k \ge
  1/{(12\sqrt{\kappa})}^k\) so that
  \[r(G) \ge 1/(12\sqrt{\kappa}{(1+{r(G)}^{q-1})}^{1/k}) \ge
    1/(24\sqrt{\kappa}),\]
it follows that \(r(H)/r(M) > 1\) for
  sufficiently small \(c\), completing the proof of the theorem.
\end{proof}

The conclusion is about radii of convergence of two power series;
application of the Cauchy-Hadamard theorem, together with one
additional hypothesis, gives the stronger result:

\begin{theorem}[Pattern Theorem;~\cite{Bender1992}]
  Suppose all of the hypotheses of Theorem~\ref{thr:weakpattern} and
  additionally that \(\MapClass\) grows smoothly, \textit{i.e.}\ that
  \(\lim_{n\to\infty}{m_n^{1/n}}\) exists. Then there exists constants
  \(c > 0\) and \(d < 1\) and \(N > 0\) so that for all \(n \ge N\),
  \[
    \frac{h_n}{m_n} < d^n.
  \]
  \textit{I.e.}, the pattern \(P\) is \emph{ubiquitous}.
  \label{thr:strongpattern}
\end{theorem}

Because of Euler's formula, the number of vertices, edges, or faces in
a link shadow or planar quadrangulation is entirely determined by
choosing any one cardinality. Hence, we can size shadows and diagrams
by the number of vertices and still keep the above results.

\subsection{Smooth growth}
\label{sec:smooth-growth}

In order for the (stronger) corollary to hold, we require that the
class of diagrams \emph{grow smoothly}, \textit{i.e.}\ that (for \(m_n =
|\MapClass_n|\)) the limit
\[
  \lim_{n\to\infty}{m_n^{1/n}}
\]
exists.

The key in proving this behavior is the strengthening of Fekete's
lemma~\cite{Rensburg2000,Wilker1979} reproduced below:

\begin{untheorem}[Wilker-Whittington 1979~\cite{Wilker1979}]
  Let \(f(m)\) be a function with
  \(\lim_{m\to\infty}{m^{-1}f(m)} = 1\) and \(\{a_n\}\) be a series.
  Suppose that \(a_na_m \le a_{n+f(m)}\) for all \(n,m\) sufficiently
  large and that \(a_n^{1/n}\) is bounded above (for all of our cases
  we have the trivial bound for all planar maps of \(144\)).
  Then
  \[ \lim_{n\to\infty}{{a_n}^{1/n}} = \limsup_{n\to\infty}{{a_n}^{1/n}} = 1/r < \infty \]
  where \(r\) is the radius of convergence of the series
  \(\sum_{n=1}^{\infty}{a_nz^n}\) and furthermore \(a_n \le (1/r)^{f(n)}\).
\end{untheorem}

For \(f(m) = m\) we recover the usual result. Observe as well that
functions \(f(m) = m + k\) satisfy the hypothesis for fixed \(k\).

%%% Separate out the applications into their own section
\section{Applications to classes of rooted diagrams}
\label{sec:appltodiagrams}

\subsection{The Pattern Theorem for rooted link diagrams}
\label{sec:patthmlinks}

As noted, the classes \(\LinkDia\) of rooted link shadows and
\(\PrimeLinkDia\) of rooted prime link shadows have been counted
exactly; they both hence grow \emph{smoothly}. Pattern theorems for
these two classes hence follow given satisfactory tangles \(P\) and
attachment operations.

%%% Two figures of tangles: \Phi = -0- and X = =0=
\begin{figure}[hbtp]
  \centering
  \begin{subfigure}[hbtp]{0.3\linewidth}
    \centering
        \begingroup
        \input{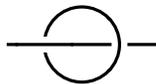}%
        \endgroup

    \caption{A tangle \(\Phi\) which can be attached into diagrams
      in \(\LinkDia\) by connect sum.}
    \label{fig:addcomptangle}
  \end{subfigure}\hfil
  \begin{subfigure}[hbtp]{0.3\linewidth}
    \centering
        \begingroup
        \input{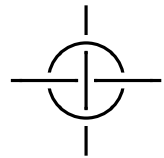}%
        \endgroup

    \caption{A tangle \(X\) which can be attached into diagrams in
      \(\PrimeLinkDia\) by crossing replacement.}
    \label{fig:addprimelinktangle}
  \end{subfigure}
  \caption{Examples of tangles for attachment to link diagrams.}
\end{figure}

\subsubsection{Rooted link diagrams} For the class of all rooted link
shadows and \emph{any} two-leg-prime \(2\)-tangle \(P\), the connect
summation with \(P\) is a viable attachment, \textit{c.f.}\
Proposition~\ref{thm:edgerep_viable}. Hence we conclude that
asymptotically almost surely any two-leg-prime tangle \(P\) is
contained linearly often in an arbitrary link diagram \(D\).

From this we see that:
\begin{proposition}
  Almost every rooted link diagram has more than one
  component. Additionally, almost every rooted link diagram is neither
  the unknot nor a split link of unknots. In other words there exists
  \(d < 1\) and \(N > 0\) so that for \(n \ge N\),
  \[
    \Prb(\text{a rooted link diagram \(L\) with \(n\) crossings has
      one component}) < d^n,
  \]
  \[
    \Prb(\text{a rooted link diagram \(L\) with \(n\) crossings
      represents the unknot}) < d^n,
  \]
  and
  \[
    \Prb\left(
      \begin{array}{c}
        \text{a rooted link diagram \(L\) with \(n\)
          crossings} \\
        \text{represents a split link of
          unknots}
      \end{array}
\right) < d^n.
  \]
\end{proposition}
\begin{proof}
  Consider the attachment of tangle \(\Phi\)
  in~\ref{fig:addcomptangle}. This attachment satisfies the hypotheses
  of Theorem~\ref{thr:strongpattern} and the class of rooted link
  diagrams grows smoothly, so there exists \(c > 0\), \(d < 1\), and
  \(N > 0\) so that for all \(n \ge N\),
  \[
    \Prb{(\text{a rooted link diagram \(L\) with \(n\) crossings
        contains \(\le cn\) copies of \(\Phi\)})} < d^n.
  \]
  Now, any of the conditions enumerated in the proposition would imply
  that the rooted link diagram with \(n\) crossings contains precisely
  no copies of the tangle \(\Phi\); but this condition happens with
  probability less than \(d^n\), providing the result.
\end{proof}

\subsubsection{Rooted prime link diagrams} In the case of rooted
\emph{prime} link diagrams, however, connect summation with \emph{any}
nontrivial 2-tangle \(P\) immediately removes the prime
condition. Hence in this case we must use crossing replacement with
4-tangles that are prime, which is viable by
Proposition~\ref{thm:xrep_viable}. We can see then that:
\begin{proposition}
  Almost all rooted prime link diagrams consist of more than one
  component. Additionally, almost all rooted prime link diagrams are
  neither unknots nor split links of unknots.
\end{proposition}
\begin{proof}
  Consider the tangle \(X\) in Figure~\ref{fig:addprimelinktangle}. For a
  rooted prime link diagram to satisfy any of the claims above, it
  would have to not contain a copy of \(X\). However, \(X\) can be
  attached by crossing replacement into the class of rooted prime link
  diagrams in a way satisfying the hypotheses of
  Theorem~\ref{thr:strongpattern}. Furthermore, the class of rooted
  prime link diagrams grows smoothly, leaving the result.
\end{proof}

\subsection{The Pattern Theorem for rooted knot diagrams}
\label{sec:patthmknots}

Unlike classes of rooted link diagrams, \(\KnotDia\) and its
subclasses often do not have exact counting formulas. Hence in these
cases we not only must provide satisfactory attachment schemes for
tangles, we must additionally prove smooth growth of counting
sequences.

\subsubsection{Rooted knot diagrams}
\label{sec:rootedkdpattern}

Provided that the tangle \(P\) consists of precisely one component,
Proposition~\ref{thm:edgerep_viable} the class \(\KnotDia\) of all
rooted knot diagrams is closed under connect sum with \(P\).

So it only remains to prove that the classes of knot diagrams
grows smoothly for pattern theorem results to follow. It suffices to
prove that the class of rooted knot \emph{shadows} grow smoothly as there
exist precisely $2^n$ rooted knot diagrams for each rooted knot shadow
of $n$ crossings.

\begin{theorem} The class \(\KnotShad\) of rooted knot shadows grows
  smoothly. That is, the limit \(\lim_{n\to\infty}{k_n^{1/n}}\) exists
  and is equal to \(1/r(K)\).

  Additionally, the class of reduced rooted knot shadows grows smoothly.
  \label{thr:smoothgrowth}
\end{theorem}

\begin{proof}
  Observe that \(k_nk_m \le k_{n+m}\) for any pair \(n, m\) as
  given any two knot shadows of size \(n,m\), say \(K_1\) and \(K_2\)
  respectively, we have the injection defined as in
  Figure~\ref{fig:legsum_example} where the shadows are joined
  end-to-end. As \(n\) and \(m\) are known quantities, the separating
  edge between \(K_1\) and \(K_2\) in the product is determined, and
  the inverse is well defined on the image.

  \begin{figure}[h!]
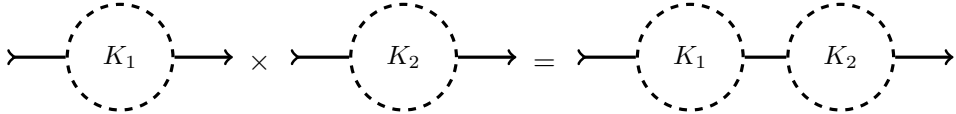
  \centering
    \[
    \begin{aligned}
        \begingroup
        \input{#twoleg_k1}%
        \endgroup
    
    \end{aligned} \times
    \begin{aligned}
        \begingroup
        \input{#twoleg_k2}%
        \endgroup
    
    \end{aligned} =
    \begin{aligned}
        \begingroup
        \input{#knotcomp}%
        \endgroup
    
    \end{aligned}
    \]
    \caption{The injection from \(\KnotShad_n \times \KnotShad_m\)
      into \(\KnotShad_{n+m}\). The edge between \(K_1\) and \(K_2\)
      is the edge connecting the \(n\)th and \((n+1)\)th crossings (by
      traversal order) and provides a well-defined inverse.}
    \label{fig:legsum_example}
  \end{figure}

  The result then follows by the above result on
  super-multiplicative sequences, with \(f(m) = m\). This proof works for both
  the case of general knot diagrams, as well as the case of reduced rooted
  diagrams, as connect summation does not introduce any new isthmi.
\end{proof}

\begin{proposition}
  Almost every rooted (general or reduced) knot diagram is knotted. Furthermore,
  almost every rooted (general or reduced) knot diagram contains any
  1-component, prime 2-tangle $P$ ``linearly often:'' For any such prime
  2-tangle \(P\), there exists \(N \ge 0\) and constants \(d < 1\), \(c > 0\) so
  that for \(n \ge N\),
  \[
    \Prb(\text{a knot diagram \(K\) contains \(\le cn\) copies of
      \(P\) as connect summands}) < d^n.
  \]

  \label{thm:patternthm}
\end{proposition}
\begin{proof}
  The first statement will follow immediately from the second, given a
  prime 2-tangle corresponding to a prime knot diagram which is not an
  unknot such as the prime trefoil tangle $T$ in
  Figure~\ref{fig:trefcs}. The second is a corollary of
  Theorems~\ref{thr:strongpattern} and~\ref{thr:smoothgrowth}: Let
  \(P\) be a prime 2-tangle which can be found as a connect summand of
  a knot diagram (\textit{i.e.}, it has one link component). If \(m_n\) is the
  number of knot diagrams, then there exists \(c > 0\), \(d < 1\), and
  \(N > 0\) so that for all \(n \ge N\), \(\frac {h_n}{m_n} < d^n\),
  where \(h_n\) is the number of knot diagrams which contain at most
  \(cn\) copies of \(P\) as connect summands. This ratio is precisely
  the probability in the second statement.

  The proof follows for reduced rooted knot diagrams precisely the same.
\end{proof}

\subsubsection{Rooted prime knot diagrams}
\label{sec:rootedprimekdpattern}

We can also show that, given a prime 4-tangle \(P\) with reasonable assumptions,
we can obtain a pattern theorem for \(P\) inside of the class \(\PrimeKnotShad\)
of prime knot diagrams. The viable attachment here is crossing replacement
discussed earlier.

It does however remain to be shown that the class \(\PrimeKnotShad\) grows
appropriately smoothly. We only need apply a slight modification to
the argument prior.

\begin{theorem}
  The class \(\PrimeKnotShad\) of prime knot shadows grows smoothly.
\end{theorem}

\begin{proof}
  For any \(n > 3\) are assured that a prime diagram will have
  \emph{no} monogons, and we can define a composition on prime
  diagrams as in Figure~\ref{fig:primeshadcomp}.

  \begin{figure}[hbtp!]
    \centering
    \[
    \begin{aligned}
        \begingroup
        \input{#twoleg_k1prime}%
        \endgroup
    
    \end{aligned} \times
    \begin{aligned}
        \begingroup
        \input{#twoleg_k2prime}%
        \endgroup
    
    \end{aligned} =
    \begin{aligned}
        \begingroup
        \input{#twoleg_xprime}%
        \endgroup
    
    \end{aligned}
    \]
    \caption{The composition \(\times\) for prime shadows from \(\PrimeKnotShad_n
      \times \PrimeKnotShad_m\) into \(\PrimeKnotShad_{n+m+2}\).}
    \label{fig:primeshadcomp}
  \end{figure}

  % . Let \((\sigma, \tau)\) be the of
  % the shadow \(D\) as a combinatorial map. Then define the arcs
  % \(g_1,...,g_{10}\) as follows: \(g_1 = \sigma(a)\) is the arc
  % counter-clockwise of \(a\), and \(g_2 = \tau(g_2)\) is the other arc in the
  % same edge. Then let \(g_3 = \sigma(g_2)\), \(g_4 = \sigma(g_3)\), and \(g_5 =
  % \sigma(g_4)\) be the three other arcs counter-clockwise about the same vertex
  % as \(g_2\).

  \begin{figure}[hbtp!]
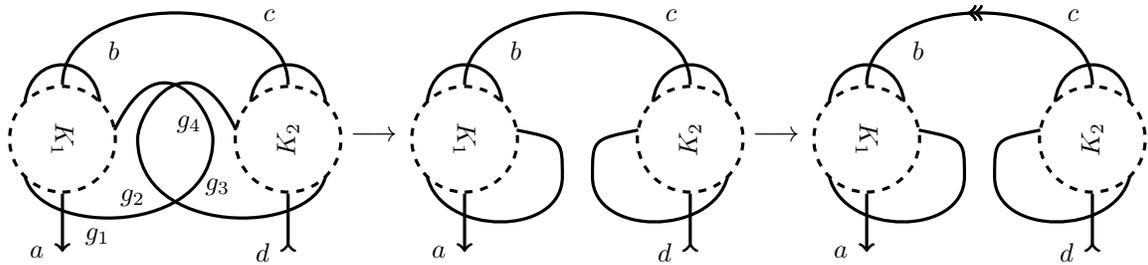

    \centering
    \[
      \begin{aligned}
        \begingroup
        \input{#twoleg_xprime}%
        \endgroup
    
      \end{aligned}
      \longrightarrow
      \begin{aligned}
        \begingroup
        \input{#twoleg_xprimerii}%
        \endgroup
    
      \end{aligned}
      \longrightarrow
      \begin{aligned}
        \begingroup
        \input{#twoleg_xprimechop}%
        \endgroup
    
      \end{aligned}
    \]
    \caption{The inverse of the composition \(\times\) for prime shadows. If the
      map is in the image of \(\times\), the arc \(g_4\) uniquely determines the
      bigon to remove, and the disconnecting edge \((bc)\) is determined.}
    \label{fig:primeshadcompinv}
  \end{figure}
  
  This operation adds precisely two vertices and is invertible on its image. We
  will define the inverse of a diagram \(D\) to be the diagram obtained by
  performing a ``flat Reidemeister II move'' to delete the bigon containing the
  arc labeled \(g_4\) and then by deleting the edge \((bc)\) between vertices
  indexed \(m\) and \(m+1\) by traversal order starting at the tail \(d\),
  \emph{if such a diagram exists and is valid}. This process is depicted in
  Figure~\ref{fig:primeshadcompinv}. This inverse proves that
  \[ p_np_m \le p_{n+f(m)},\]
  where \(f(m) = m+2\). This, together with the above theorem on
  super-multiplicativity completes the proof.
\end{proof}

Hence,

\begin{proposition}
  Almost every rooted prime knot diagram is knotted. Furthermore, almost every
  rooted prime knot diagram contains any 2-component, prime 4-tangle $P$
  ``linearly often:'' For any such prime 4-tangle \(P\), there exists \(N \ge
  0\) and constants \(d < 1\), \(c > 0\) so that for \(n \ge N\),
  \[
    \Prb(\text{a prime knot diagram \(K\) contains \(\le cn\) copies of
      \(P\)}) < d^n.
  \] \label{thm:primepatternthm}
\end{proposition}
\begin{proof}
  \begin{figure}[hbtp!]
    \centering
        \begingroup
        \input{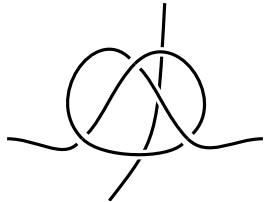}%
        \endgroup
    
    \caption{A prime 4-tangle whose insertion into a diagram guarantees
      knottedness, as it adds a \(3_1\) connect sum component.}
    \label{fig:primeknottangle}
  \end{figure}
  The proof of the latter statement is nearly identical to that of
  Proposition~\ref{thm:patternthm}, except with crossing replacement instead of
  connect summation. A 4-tangle whose insertion guarantees that a diagram be
  knotted is depicted in Figure~\ref{fig:primeknottangle}, from which the former
  statement follows.
\end{proof}

\section{Asymmetry of diagrams}
\label{sec:asymmetry}

It is a theorem of Richmond and Wormald~\cite{Richmond19951} that
``almost all maps are asymmetric'', in classes of maps which have a
pattern theorem. Indeed, this theorem applies to classes of shadows
considered in this paper. Notice then that classes of diagrams then
too must be asymmetric as the decoration imposes additional
constraints on symmetry.

Say a tangle shadow \(P\) is \emph{free} in a class of rooted knot
shadows \(\mathscr C\) if any knot shadow obtained by removing a copy
\(P_1\) of \(P\) from a shadow \(K\) and re-attaching \(P\) in any
fashion such that the exterior legs of \(P\) match up with the loose
strands of \(K \setminus P_1\) is in \(\mathscr C\).

We then restate the theorem as it pertains to shadows:

\begin{untheorem}[Richmond-Wormald 1995]
  Let \(\mathscr C\) be a class of rooted shadows. Suppose that there is a tangle shadow \(P\) such that for all shadows in \(\mathscr C\), all copies of \(P\) are pairwise disjoint and such that \(P\)
  \begin{enumerate}
  \item has no reflective symmetry in the plane,
  \item satisfies the hypotheses for the Pattern Theorem~\ref{thr:strongpattern} for \(\mathscr C\), and
  \item is free in \(\mathscr C\)
  \end{enumerate}
  Then the proportion of \(n\)-crossing shadows in \(\mathscr C\) with
  nontrivial automorphisms (that need not preserve neither root nor
  orientation of the underlying surface) is exponentially small.
\end{untheorem}

It has been suggested without proof
in~\cite{Coquereaux16,Schaeffer2004} that
classes of knot shadows are almost surely asymmetric. We can now prove
these results for \(\FlatKnotDia\) and \(\PrimeKnotShad\) by providing
appropriate tangles \(P\).

\begin{proposition}
  The proportion of knot shadows in \(\KnotShad\) with nontrivial
  automorphisms is exponentially small. Additionally, the proportion
  of knot diagrams in \(\KnotDia\) with nontrivial automorphisms is
  exponentially small.

  This is true for reduced knot shadows as well.
\end{proposition}

\begin{proof}

\begin{figure}[h!]
  \centering
        \begingroup
        \input{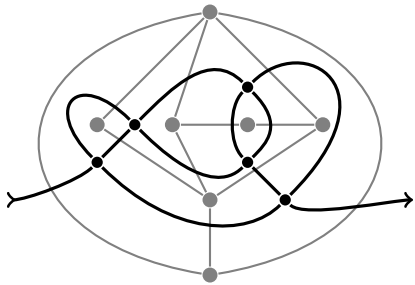}%
        \endgroup
    
  \caption{The (prime, reduced) 2-tangle \(P\) (black) and its dual
    quadrangulation (gray).}
  \label{fig:asymm2tangle}
\end{figure}
Let \(P\) be the 2-tangle in Figure~\ref{fig:asymm2tangle}. \(P\) is a viable
candidate for the Pattern Theorem for knot shadows under connect sum attachment.
Furthermore \(P\) has no reflective symmetry by inspection, and any of the ways
of attaching \(P\) keep the object in the class of knot shadows (as \(P\) is a
2-tangle consisting of one component). This proves that knot shadows are
asymmetric. The proof for reduced knot shadows is identical.
\end{proof}

\begin{proposition}
  The proportion of prime knot shadows in \(\PrimeKnotShad\) with
  nontrivial automorphisms is exponentially small. Additionally, the
  proportion of knot diagrams in \(\PrimeKnotDia\) with nontrivial
  automorphisms is exponentially small.
  \label{thm:primeknotasymm}
\end{proposition}

\begin{proof}
  Take \(P\) as in Figure~\ref{fig:primeknotasymm}.
  \begin{figure}[h!]
    \centering
        \begingroup
        \input{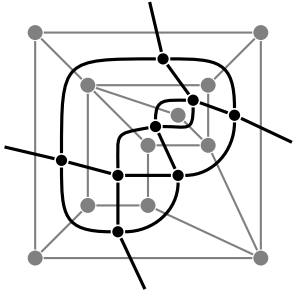}%
        \endgroup
    
    \caption{Choice of 4-tangle \(P\) for proving that prime knot shadows are
      asymmetric.}
    \label{fig:primeknotasymm}
  \end{figure}
  Then \(P\) consists of exactly two link components and is of
  red-blue-red-blue type; any way of replacing a vertex in a knot
  shadow with a 4-tangle curve of red-blue-red-blue type keeps the
  number of link components constant and hence is free. Furthermore,
  \(P\) satisfies the hypotheses for Theorem~\ref{thm:patternthm} together
  with crossing replacement.
\end{proof}
Application of the above theorems provides us with the following
corollary which enables us to transfer any asymptotic numerical
results or sampling results on rooted diagrams to unrooted diagrams.

\begin{corollary}
  Let \(L\) be a uniform random variable taking values in the space
  \(\KnotDia_n\) or \(\LinkDia_n\). Then there exist constants
  \(C, \alpha > 0\) so that \(\Prb(\Aut L \ne 1) < Ce^{-\alpha
    n}\). Hence, rooted diagrams behave like unrooted diagrams and
  there are a.a.s.\ $4n$ rooted diagrams to each unrooted diagram.
\end{corollary}

Indeed, link diagrams with \(n\) vertices are dual to quadrangulations
with \(n+2\) faces; there are \(n+2\) ways of choosing the
``exterior'' root face and then \(4\) ways of rooting the edges around
this chosen face. Hence if \(\tilde \ell_n\), \(\tilde k_n\) are the
counts of unrooted link or knot diagrams we have that in the limit,
\[ \tilde\ell_n \underset{n\to\infty}{\sim} \frac{\ell_n}{4(n+2)}
\text{ and } \tilde k_n \underset{n\to\infty}{\sim}
\frac{k_n}{4(n+2)}.\]

% \begin{corollary}
%   A random knot or link diagram has the pattern theorem. Namely, a
%   random knot diagram is almost surely composite and almost surely
%   knotted, and a random link diagram is almost surely not a knot
%   diagram.
% \end{corollary}

\section{Some numerical results}
\label{sec:randres}

\begin{figure}[hbtp!]
  \centering
        \begingroup
        \input{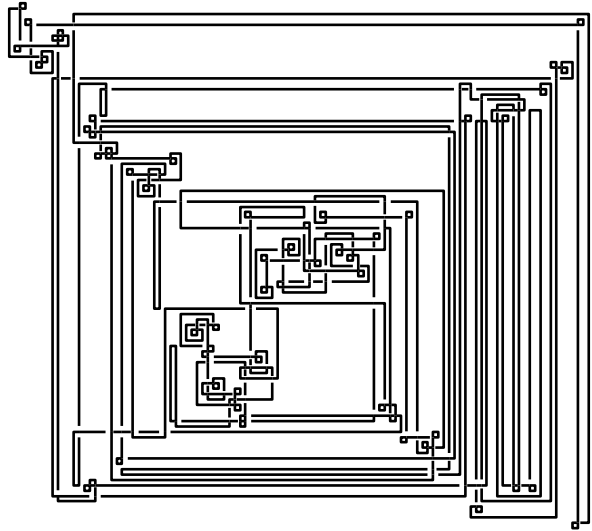}%
        \endgroup
    
  \caption{A randomly sampled knot diagram with 150 crossings,
    presented in a particular orthogonal projection. The drawing
    itself was produced using Culler and Dunfield's \texttt{PLink}
    software~\cite{PLink}.}
  \label{fig:bigrandomdiagram}
\end{figure}

Gilles Schaeffer's \texttt{PlanarMap} software~\cite{SchaefferPlanarMap} is able
to uniformly sample rooted link shadows and rooted prime link shadows by using a
bijection between shadows and objects called \emph{blossom
  trees}~\cite{Schaeffer2004,Schaeffer1997}. We have implemented this
functionality into \texttt{plCurve}~\cite{PlCurve} in order to uniformly sample
rooted (optionally prime) link diagrams. By rejecting link diagrams with more
than one component, we are able to restrict and, with some effort, uniformly
sample rooted (optionally prime) knot diagrams. An example of such a knot diagram
with 150 crossings is presented in Figure~\ref{fig:bigrandomdiagram}.

One may be concerned that the ``asymptotic'' behavior proved in the prior
section only applies to knot diagrams with an absurd number of crossings (in the
sense that no physical knot should be expected to be so complicated). However,
exact and numerical results show that this behavior is attained very quickly.
For example, almost all 10-crossing knot diagrams have no nontrivial
automorphisms! This is exhibited in Figure~\ref{fig:knotauts}.
\begin{figure}[hbtp!]
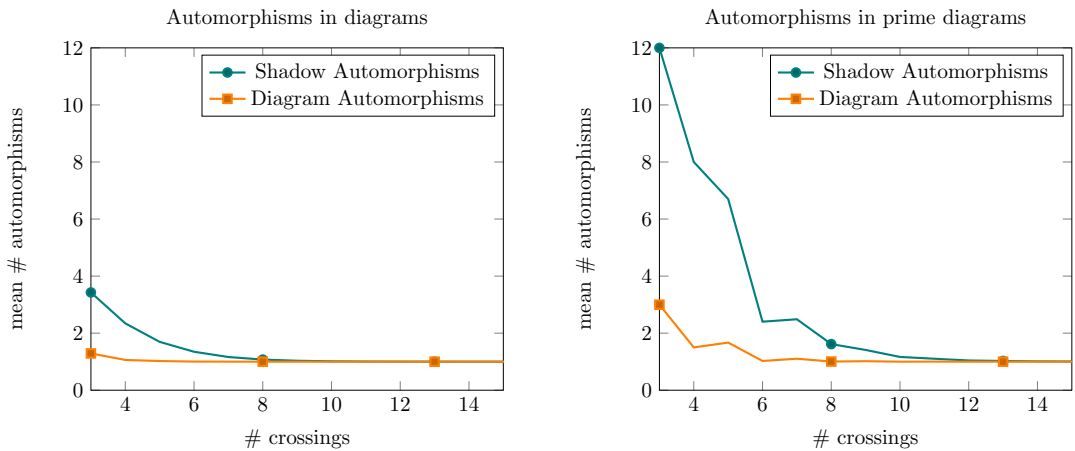

  \centering
  \begin{subfigure}{0.45\linewidth}
    \centering
        \begingroup
        \input{#allknotauts}%
        \endgroup
    
    \label{fig:allknotauts}
  \end{subfigure}
  \begin{subfigure}{0.45\linewidth}
    \centering
        \begingroup
        \input{#primeknotauts}%
        \endgroup
    
    \label{fig:primeknotauts}
  \end{subfigure}
  \caption{Automorphisms in knot diagrams decrease exponentially
    quickly.}
  \label{fig:knotauts}
\end{figure}

We will now consider relative appearances of various knot types among knot
diagrams. Knots were classified by their HOMFLY polynomial~\cite{Freyd85}, which
distinguishes knots of the types studied (\(0_1, 3_1, 4_1, 5_1, 5_2, 3_1\#3_1\))
for diagrams of low crossing number. For larger diagrams, HOMFLY is not
necessarily a complete invariant. However, these clashes are expected to be rare
and hence immaterial for the analysis that follows. Furthermore, it is believed
at least that the HOMFLY polynomial distinguishes the unknot; a counterexample
would disprove the fabled ``Jones Conjecture~\cite{Kauffman87},'' and would be a
monumental discovery by itself.

We note additionally that our data ignores chirality of knot types: The knot
types \(3_1, 5_1, 5_2\) admit both left- and right-handed chiralities (which are
distinguished by HOMFLY), but the chirality is ignored as the probability of the
two mirror-images are identical in the diagram model. Finally, under the
composite knot type \(3_1\#3_1\) is data for both the granny knot, its mirror
image, and the square knot, all of whom have equal probability as well.

This allows us to expand on our results in~\cite{Cantarella2015}. Although we no
longer have \emph{exact} numerics, we can break past the 10-crossing barrier
with relative ease and good precision. Compare the chart in
Figure~\ref{fig:knotprobdata} to Table III in~\cite{Cantarella2015}. 

\begin{figure}[hbtp!]
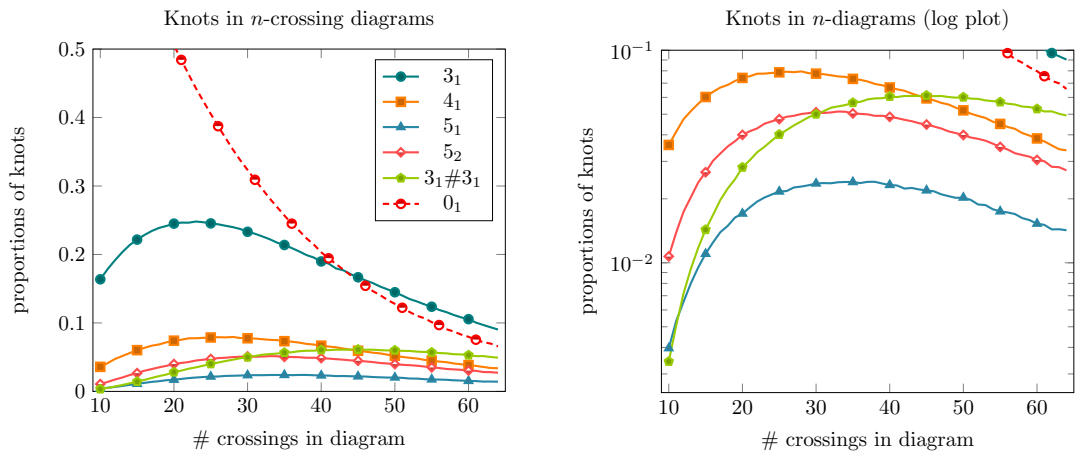

  \centering
  \begin{subfigure}{0.45\linewidth}
    \centering
        \begingroup
        \input{#kgrow1}%
        \endgroup
    
    \label{fig:kgrow1}
  \end{subfigure}
  \begin{subfigure}{0.45\linewidth}
    \centering
        \begingroup
        \input{#kgrow2}%
        \endgroup
    
    \label{fig:kgrow2}
  \end{subfigure}

  \caption{Probabilities of some typical knot types as the number of
    crossings varies from \(n=10\) to \(n=64\). 500,000 samples were
    attempted for each \(n\), although samples were discarded if a
    knot diagram was not produced after 49 attempts. There were,
    \textit{e.g.}\ 258,377 samples for \(n=64\).}
  \label{fig:knotprobdata}
\end{figure}

Of note is the exponential decay of the proportion of unknots \(0_1\) as the
number of crossings grow large. In fact, data shows that the number of (both
left- and right-) trefoils \(3_1\) exceeds that of the unknot \(0_1\) at around
\(n = 44\) crossings. Additionally, the trend for the square and granny knots,
labeled \(3_1\#3_1\), has it decaying at a far slower rate than any of the prime
knot types. This is in line with conjectures for models of random knots that the
growth rates of knot types \(K\) should accrue an additional polynomial growth
factor of \(n\) for each prime knot type in the decomposition of
\(K\)~\cite{Rensburg2011}. Also of note is the so-called ``\(5_1, 5_2\)
inversion,'' where the knot \(5_2\) is more likely than the torus knot \(5_1\).
This was exhibited in the earlier exact data of~\cite{Cantarella2015} and
persists through the entire scope of the data gathered.

\begin{figure}[hbtp!]
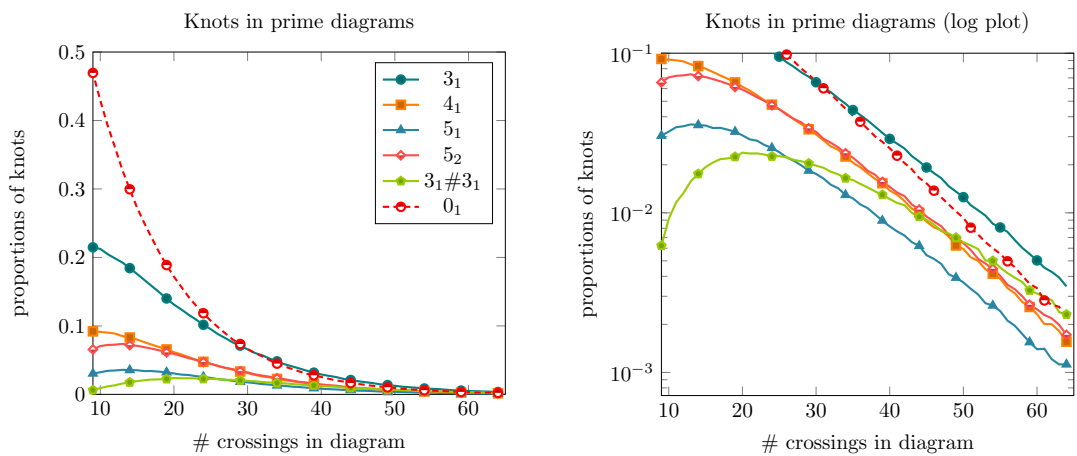

  \centering
  \begin{subfigure}{0.45\linewidth}
    \centering
        \begingroup
        \input{#pkgrow1}%
        \endgroup
    
    \label{fig:pkgrow1}
  \end{subfigure}
  \begin{subfigure}{0.45\linewidth}
    \centering
        \begingroup
        \input{#pkgrow2}%
        \endgroup
    
    \label{fig:pkgrow2}
  \end{subfigure}

  \caption{Probabilities of some typical knot types in prime diagrams as the
    number of crossings varies from \(n=3\) to \(n=64\). 500,000 samples were
    attempted for each \(n\), although samples were discarded if a knot diagram
    was not produced after 49 attempts. There were, \textit{e.g.}\ 496,151
    samples for \(n=64\).}
  \label{fig:primeknotprobdata}
\end{figure}

Our rejection sampling methods additionally apply in the case of sampling knot
diagrams with different graph-theoretic constraints. For instance, it extends to
provide a sampler for prime knot diagrams, which have underlying map structure
which is \(4\)-edge connected. As dictated by
Proposition~\ref{thm:primeknotasymm}, the number of asymmetries of prime knot
diagrams tends to zero exponentially quickly as in Figure~\ref{fig:knotauts}. We
also present data for the proportions knots in prime diagrams in
Figure~\ref{fig:primeknotprobdata}. While the precise numerics differ, notice
that prime diagrams exhibit the same interesting behavior as generic diagrams
discussed prior. Namely, (1) unknots are exponentially rare and become less
prevalent than trefoils (albeit sooner in the case of prime diagrams, occurring
at around \(n=30\) crossings), (2) the composite square and granny knots show
slower decay than knots of prime type, and (3) the \(5_1, 5_2\) inversion.

\begin{figure}
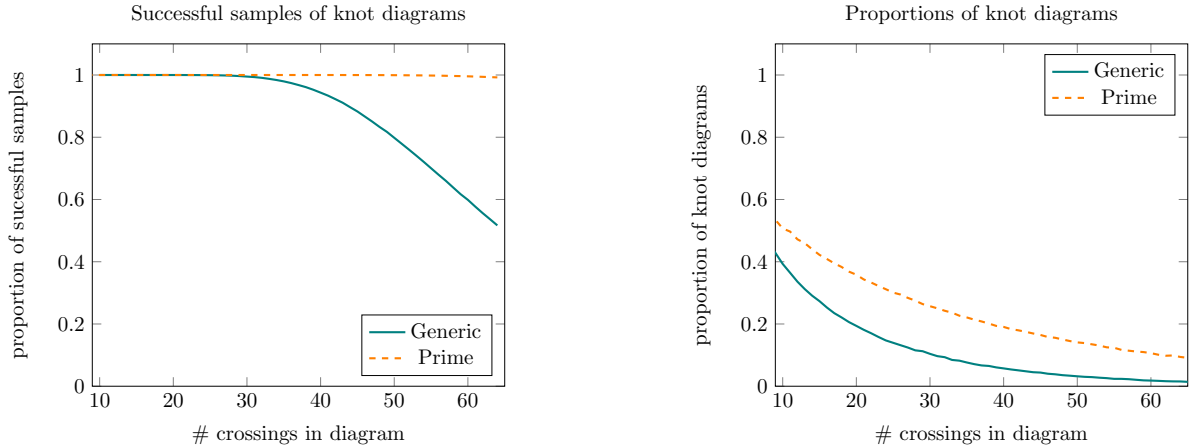

  \centering
  \begin{subfigure}[t]{0.45\linewidth}
    \centering
        \begingroup
        \input{#samplecounts}%
        \endgroup
    
    \caption{Number of successful samples for diagrams in gathering the above
      data. For each of 500,000 samples, a maximum of 49 attempts is made to
      generate a knot diagram, and if unsuccessful, the sample is discarded.}
    \label{fig:samplecounts}
  \end{subfigure} \hfill
  \begin{subfigure}[t]{0.45\linewidth}
    \centering
        \begingroup
        \input{#knotlinkcounts}%
        \endgroup
    
    \caption{Ratio of knot diagrams to link diagrams in the generic and prime
      cases. Data was gathered by sampling 100,000 link diagrams and counting
      the number of samples with precisely one link component.}
    \label{fig:knotlinkcounts}
  \end{subfigure}
  \caption{The exponential decay of knot diagrams among link diagrams makes
    rejection sampling more difficult as the number of crossings increases.}
  \label{fig:samplecountingdata}
\end{figure}

An observation from running the above experiments is that knot diagrams appear
to be more prevalent in prime diagrams than generic knot diagrams! Indeed,
nearly all rejection samples of prime knot diagrams successful (see
Figure~\ref{fig:samplecounts}) as opposed to the 50\% success rate of sampling
generic diagrams of \(64\) crossings. A more precise computation of the
proportion of knot diagrams to link diagrams is exhibited in
Figure~\ref{fig:knotlinkcounts}, where we can see that there is indeed a far
greater proportion of knot diagrams in the prime diagram case to the generic, at
least up to \(65\) crossings.

Hence it is worth summarizing these interesting properties of the prime knot
diagram model. First, prime link diagrams admit exact enumeration like general
link diagrams and are hence easy to sample. In addition, there is a higher
success rate of sampling knot diagrams inside of prime diagrams as opposed to
general diagrams. Second, the difficulty of inserting ``boring'' structure into
prime knot diagrams suggests that there is far more variety in the knot types
which arise than in the general case. So, prime diagrams are a natural class to
find diagrams representing exotic knot types with large minimum crossing number.
Some evidence to this is the fact that \emph{alternating}, prime knot diagrams
are minimal in that they have the fewest number of crossings over all diagram
representations of their knot type. Finally, the smaller size of the class of
prime diagrams (as opposed to the general case) suggests that fewer samples are
required to have a better statistical understanding of the space.

As mentioned before, we are interested in how the knot diagram model
differs from other models of random knotting. Specifically, we note
that if there is universality of knot probability ratios in lattice
models~\cite{Rensburg2011}, our data suggests that it does not extend
to the knot diagram model. For instance, Rechnitzer and Rensburg
speculate that, for all self-avoiding polygon models of knotting, that
the ratio of \(3_1\) knots to \(4_1\) knots should be \(28(\pm
1)\)-to-1, while our data (Figure~\ref{fig:trefeightratio}) suggests that in the knot diagram model, the
ratio is about \(3(\pm 1)\)-to-1.

\begin{figure}[hbtp!]
  \centering
        \begingroup
        \input{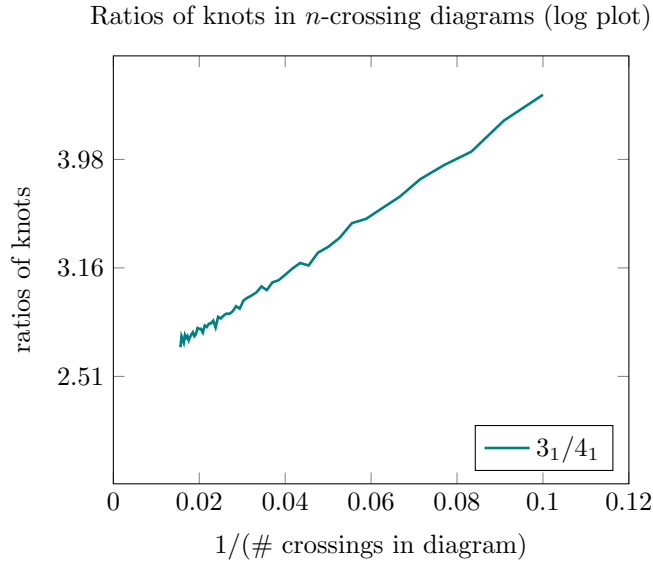}%
        \endgroup
    
  \caption{Experiments suggest that the limiting ratio of \(3_1\) to
    \(4_1\) in the knot diagram model is about \(3(\pm 1)\). Data is
    drawn from the same experiment as Figure~\ref{fig:knotprobdata}.}
  \label{fig:trefeightratio}
\end{figure}

\section{Conclusion}
\label{sec:conclusion}

We have shown a pattern theorem for knot diagrams, which is one of the most
important results when working with models of random knots. As a consequence, we
were able to show that like most other models of knotting used in studying
physical polymers, large diagrams are asymptotically almost surely knotted. This
suggests that the diagram model may describe well certain types of polymer
knotting, although the question of which remains open.

Another consequence of the pattern theorem---that almost all knot diagrams are
asymmetric---greatly simplifies the sampling of knot diagrams uniformly. A
result of this is a rejection sampler for rooted knot diagrams (of several
different graph-theoretic types), which is implemented in
\texttt{plCurve}~\cite{PlCurve} and publicly available. It remains open however
whether it is possible to sample directly (\textit{i.e.}\ without rejection)
from the class of knot diagrams, eliminating the inefficiencies caused by the
exponential decay of knot diagrams in link diagrams. As progress towards this,
together with Rechnitzer we have constructed a Markov Chain Monte Carlo
algorithm for sampling knot diagrams with a distribution limiting on the uniform
distribution across diagrams with fixed crossing number~\cite{Chapman2016:mc}.

Given the similarities to other models of random models, we mention a driving
theme behind studying the diagram model: The ease of computing knot invariants
such as the HOMFLY polynomial~\cite{Freyd85} together with the more topological
definition of diagrams suggests that we can prove new theorems for the knot
diagram model which remain conjectures in classical models of random knotting.
Similarities with space curve models then suggest that there is a way to
transfer these results backwards to their original models.

\section{Acknowledgments}
\label{sec:acknowledgements}

The author is extremely grateful to his advisor Jason Cantarella, for
his support, advice, and for introducing him to the knot tabulation
project (alongside Matt Mastin) and suggesting he prove
Theorem~\ref{thm:knotted}. The author is also grateful to the summer
school on applied combinatorics at the University of Saskatchewan and
CanaDAM thereafter. The author also is indebted to funding from the
NSF (grant DMS-1344994 of the RTG in Algebra, Algebraic Geometry, and
Number Theory, at the University of Georgia), PIMS, the Simons Center,
and the AMS through which he was able to introduce his work to
others. The author would like to thank Gary Iliev greatly for bringing
reference~\cite{Rensburg2000} to his attention.

The author is grateful for conversations with Julien Courtiel,
Elizabeth Denne, Chris Duffy, Chaim Even-Zohar, \'Eric Fusy, Rafa\l{}
Komendarczyk, Neal Madras, Kenneth Millett, Marni Mishna, Erik Panzer,
Jason Parsley, Eric Rawdon, Andrew Rechnitzer, Clayton Shonkwiler,
Chris Soteros, and Karen Yeats.

An extended abstract for this paper appears in the proceedings for
FPSAC 2016.

\begingroup
\raggedright{}
\sloppy
\printbibliography{}
\endgroup

\end{document}